\documentclass[a4paper, 10pt]{amsart}
\usepackage{amsmath}
\usepackage{amssymb, color}

\usepackage[colorlinks]{hyperref}
\definecolor{linkblue}{RGB}{1,1,190}
\definecolor{citered}{RGB}{190,1,1}
\hypersetup{
	linkcolor=linkblue,
	urlcolor=linkblue,
	citecolor=citered
}

\theoremstyle{plain}
\newtheorem{theorem}{\bf Theorem}[section]
\newtheorem{proposition}[theorem]{\bf Proposition}
\newtheorem{lemma}[theorem]{\bf Lemma}
\newtheorem{corollary}[theorem]{\bf Corollary}

\theoremstyle{definition}

\newcommand{\N}{\mathbb N}
\newcommand{\Z}{\mathbb Z}

 \DeclareMathOperator{\ord}{ord}

 \DeclareMathOperator{\supp}{supp}

\renewcommand{\time}{\negthinspace \times \negthinspace}

\newcommand{\id}{{\text{\rm id}}}
\renewcommand{\t}{\, | \,}

\newcommand{\red}{{\text{\rm red}}}

\newcommand{\BF}{\text{\rm BF}}
\newcommand{\FF}{\text{\rm FF}}

\newcommand{\AAMP}{\text{\rm AAMP }}

\numberwithin{equation}{section}

\thanks{This work was supported by the Austrian Science Fund FWF (Project P36852-N)}

\begin{document}

\title[On Sets of Lengths in Monoids of plus-minus weighted  Zero-Sum Sequences]{On Sets of Lengths in Monoids of plus-minus weighted \\ Zero-Sum Sequences over Abelian Groups}

\author{Alfred Geroldinger and Florian Kainrath}

\address{University of Graz, NAWI Graz \\
Department of Mathematics and Scientific Computing \\
Heinrichstra{\ss}e 36\\
8010 Graz}

\email{alfred.geroldinger@uni-graz.at, florian.kainrath@uni-graz.at}
\urladdr{https://imsc.uni-graz.at/geroldinger}



\begin{abstract}
Let $G$ be an additive  abelian group.  A sequence $S = g_1 \cdot \ldots \cdot g_{\ell}$ of terms from $G$ is a plus-minus weighted zero-sum sequence if there are $\varepsilon_1, \ldots, \varepsilon_{\ell} \in \{-1, 1\}$ such that $\varepsilon_1 g_1 + \ldots + \varepsilon_{\ell} g_{\ell}=0$. We study sets of lengths in the monoid $\mathcal B_{\pm} (G)$ of plus-minus weighted zero-sum sequences over $G$. If $G$ is finite, then sets of lengths are highly structured. If $G$ is infinite, then every finite, nonempty subset of $\N_{\ge 2}$ is the set of lengths of some sequence $S \in \mathcal B_{\pm} (G)$.
\end{abstract}

\maketitle

\smallskip
\section{Introduction} \label{1}
\smallskip

Let $G$ be an additive abelian group. By a sequence over $G$, we mean a finite sequence of terms from $G$, where the order of terms is disregarded and repetition is allowed. We consider sequences as elements of the (multiplicatively written) free abelian monoid $\mathcal F (G)$ with basis $G$ (the multiplication of sequences in $\mathcal F (G)$ means the concatenation of sequences in combinatorial language). A sequence $S = g_1 \cdot \ldots \cdot   g_{\ell}$, with terms $g_1, \ldots, g_{\ell}$ from $G$, is a zero-sum sequence if $g_1 + \ldots + g_{\ell} = 0$ and it is a plus-minus weighted zero-sum sequence if $\varepsilon_1 g_1 + \ldots + \varepsilon_{\ell} g_{\ell}=0$ for some $\varepsilon_1, \ldots, \varepsilon_{\ell} \in \{-1, 1\}$.

Weighted zero-sum sequences (from fully weighted sequences to plus-minus weighted sequences) have been studied  in additive combinatorics since the last two decades. Among others, many of the classical zero-sum invariants (such as the Davenport constant, Gao's constant, and others) gave rise to weighted analogs (for a sample see \cite[Chapter 16]{Gr13a}, \cite{A-C-F-K-P06,HK14b, Ma-Or-Sa-Sc15,Gr-Ma-Or12a, Gr-He15a, Ma-Or-Ra-Sc16a}).

The monoid $\mathcal B (G)$ of all zero-sum sequences over $G$ is a Krull monoid and it plays a universal role in the arithmetic study of general Krull monoids. Pushed forward by this connection, algebraic and arithmetic properties of $\mathcal B (G)$ are central topics in the factorization theory of rings and monoids. The study of the monoid $\mathcal B_{\pm} (G)$ of plus-minus weighted zero-sum sequences over $G$, from an algebraic viewpoint, was initiated only a couple of years ago by Schmid and his coauthors. Algebraic topics include questions when monoids of weighted zero-sum sequences are Krull, Mori, or finitely generated. Arithmetic topics deal with questions on (various types of) Davenport constants and on  invariants controlling the structure of sets of lengths. The isomorphism problem asks whether, for given abelian groups $G_1$ and $G_2$, the monoids $\mathcal B_{\pm} (G_1)$ and $\mathcal B_{\pm} (G_2)$ are isomorphic if and only if the groups $G_1$ and $G_2$ are isomorphic (for all these topics see \cite{B-M-O-S22, Ge-HK-Zh22, F-G-R-Z24a,Me-Or-Sc25a}).

In the present paper, we study sets of lengths in monoids of plus-minus weighted zero-sum sequences.
To fix notation, let $H$ be a commutative and cancellative monoid. For an element $a \in H$, we denote by $\mathsf L_H (a) \subset \N_0$ the set of lengths of $a$, and by $\mathcal L (H) = \{ \mathsf L (a) \colon a \in H\}$ the system of sets of lengths of $H$ (for details, see Section \ref{2}).

We formulate our main result (more information on  AAMPs and on the set of minimal distances $\Delta^* \big( \mathcal B_{\pm} (G) \big)$ will be given in Section \ref{5}).

\newpage
\smallskip
\begin{theorem} \label{1.1}
Let $G$ be an abelian group.
\begin{enumerate}
\smallskip
\item If $G$ is finite, then there is a bound $M (G) \in \N_0$ such that, for every plus-minus weighted sequence $S\in \mathcal B_{\pm} (G)$, its set of lengths  $\mathsf L_{\mathcal B_{\pm} (G)} (S)$ is an \AAMP with difference $d \in \Delta^* \big( \mathcal B_{\pm} (G) \big)$ and bound $M (G)$.

\smallskip
\item If  $G$ is infinite, then, for every finite, nonempty subset $L \subset \N_{\ge 2}$ and every map $f \colon L \to \N$, there is an $S\in \mathcal B (G)$ with the following properties.
    \begin{enumerate}
    \item $S$ is squarefree in $\mathcal F (G)$.

    \item $\mathsf L_{\mathcal B_{\pm} (G)} (S) = \mathsf L_{\mathcal B (G)} (S) =  L$.

    \item $|\mathsf Z_{\mathcal B_{\pm} (G),k} (S)| \ge |\mathsf Z_{\mathcal B (G),k} (S)| \ge f (k)$ for all $k \in L$. Moreover, both inequalities are equalities for  $k > \min L$.
    \end{enumerate}
\end{enumerate}
\end{theorem}

\smallskip
Before we discuss the statements in detail, consider a transfer homomorphism $\theta \colon H \to B$ between monoids $H$ and $B$ (definitions are given in Section \ref{2}). The existence of such a homomorphism implies that  $\mathcal L (H) = \mathcal L (B)$. Krull domains, Krull monoids, and transfer Krull monoids (see the examples discussed in the survey \cite{Ge-Zh20a}) allow transfer homomorphisms to monoids of zero-sum sequences over subsets of abelian groups. Thus, sets of lengths in such monoids and domains can be studied in monoids of zero-sum sequences. Similarly, monoids of plus-minus weighted zero-sum sequences occur as target monoids of transfer homomorphisms which start, for example,   from norm monoids of orders in Galois number fields. On the other hand, (apart from a trivial exceptional case) there is no transfer homomorphism from a  monoid of plus-minus weighted zero-sum sequences to any Krull monoid (\cite[Corollary 3.5]{F-G-R-Z24a}). This demonstrates that we cannot get the above result via transfer homomorphisms from an associated result on Krull monoids.

\smallskip
The first statement of Theorem \ref{1.1} is a simple consequence of known results and it is  formulated to highlight the difference between finite and infinite groups (see Section \ref{5} for more on $\Delta^* \big( \mathcal B_{\pm} (G) \big)$). The Characterization Problem (for monoids of plus-minus weighted zero-sum sequences) asks which finite abelian groups $G_1$ have the  property that the equality of systems of sets of lengths $\mathcal L \big( \mathcal B_{\pm} (G_1) \big) = \mathcal L \big( \mathcal B_{\pm} (G_2) \big)$ implies that the groups $G_1$ and $G_2$ are isomorphic, for any finite abelian group $G_2$ (see \cite{F-G-R-Z24a, Me-Or-Sc25a}). All work in this direction is based on the   structural description given in Theorem \ref{1.1}.1.

The second statement of Theorem \ref{1.1} implies, in particular, that every finite, nonempty subset of $\N_{\ge 2}$ occurs as a set of lengths. Thus, the monoid $H = \mathcal B_{\pm} (G)$ satisfies the property
\begin{equation} \label{full}
\mathcal L \big( H \big) = \big\{ \{0\}, \{1\} \big\} \cup \big\{ L \subset \N_{\ge 2} \colon L \ \text{is finite and nonempty} \big\} \,,
\end{equation}
and hence $\mathcal L \big( \mathcal B_{\pm} (G_1) \big) = \mathcal L \big( \mathcal B_{\pm} (G_2) \big)$ for any two infinite abelian groups $G_1$ and $G_2$.

Property \eqref{full} was first proved by Kainrath   for  Krull monoids with infinite abelian class group and prime divisors in all classes (\cite{Ka99a}).
Since then Property \eqref{full} was shown to hold true for various classes of integer-valued polynomials (\cite{Fr13a, Fr-Na-Ri19a, Fa-Fr-Wi23,Fa-Wi24a}), for some classes of primary monoids (\cite[Theorem 3.6]{Go19a}), and others (\cite{Fa-Zh23a}). Furthermore, there are monoid algebras which satisfy the ascending chain condition on principal ideals (whence they are atomic) and which have the property that every (not necessarily finite) nonempty subset of $\N_{\ge 2}$ occurs as a set of lengths (\cite{Ge-Go25a}).

After the first result in \cite{Ka99a},   several distinct constructions realizing sets of lengths in Krull monoids were given (for a realization result in numerical monoids see \cite{Ge-Sc18e}).
\begin{itemize}
\item[(i)] A realization theorem for a single finite, nonempty subset $L \subset \N_{\ge 2}$ in an abstract finitely generated Krull monoid $H_L$ (\cite[Proposition 4.8.3]{Ge-HK06a}; note if $\mathcal L$ is a family of finite subsets, then all sets of $\mathcal L$ are sets of lengths in the coproduct $\coprod_{L \in \mathcal L} H_L$, and this coproduct is Krull again).

\item[(ii)] A realization theorem for a single finite, nonempty subset $L \subset \N_{\ge 2}$ in a monoid of zero-sum sequences over a finite abelian group (\cite{Sc09a}).

\item[(iii)] A realization theorem for all finite, nonempty subsets $L \subset \N_{\ge 2}$ in a monoid of zero-sum sequences over an abelian group containing an element of infinite order (\cite[Theorem 3]{B-R-S-S16}).
\end{itemize}

Our proof of Theorem \ref{1.1} follows  the  ideas given in \cite{Ka99a}. We recapitulate the whole construction, which  allows us to prove not only Property \eqref{full} for $\mathcal B_{\pm} (G)$, but the stronger Properties (a), (b), and (c) of Theorem \ref{1.1}. Moreover, this increases the readability of the present paper.

We obtain the following corollary, which   should be seen against the background of the above mentioned Characterization Problem (for related results see \cite[Theorem 4]{Wi24a}, \cite[Theorem 3.7]{Ge-Sc-Zh17b}).

\smallskip
\begin{corollary} \label{1.2}
Let $L \subset \N_{\ge 2}$ be a finite, nonempty subset. Then there are only finitely many pairwise non-isomorphic finite abelian groups $G$ such that $L \notin \mathcal L \big( \mathcal B_{\pm} (G) \big)$.
\end{corollary}

\smallskip
In Section \ref{2}, we put together the background on factorizations and on monoids of (weighted) zero-sum sequences.
A crucial step in the proof of Theorem \ref{1.1} will be done in Section \ref{3} and a further reduction step will be handled in Section \ref{4}. Finally,  we complete the proofs of Theorem \ref{1.1} and of  Corollary \ref{1.2} in Section \ref{5}. Along our way, we oftentimes  (as in Theorem \ref{1.1}) consider sequences $S \in \mathcal B (G)$ and find out that the set of lengths or even all factorizations of $S$ in $\mathcal B (G)$ are the same as  in $\mathcal B_{\pm} (G)$. Needless to say, that this is far from being true in general (for a striking difference see Proposition \ref{5.2}) but holds true only in well-constructed exceptional cases.

\smallskip
\section{Prerequisites} \label{2}
\smallskip

By a {\it monoid}, we mean a commutative, cancellative semigroup with identity element, and we use multiplicative notation throughout. For integers $a, b \in \Z$, let $[a,b] = \{ x \in \Z \colon a \le x \le b \}$ be the discrete interval between $a$ and $b$. For a set $P$, we denote by $\mathcal F (P)$ the free abelian monoid with basis $P$. An element $a \in \mathcal F (P)$ has a unique representation in the form
\[
a = \prod_{p \in P}p^{\mathsf v_p (a)} \,,
\]
where $\mathsf v_p \colon \mathcal F (P) \to \N_0$ is the $p$-adic valuation of $a$. Then $|a| = \sum_{p \in P} \mathsf v_p (a) \in \N_0$ denotes the {\it length} of $a$ and $\supp (a) = \{ p \in P \colon \mathsf v_p (a) > 0 \} \subset P$ the {\it support} of $a$.

Let $H$ be a monoid. We denote by $H^{\times}$ its group of invertible element and by $H_{\red} = \{ aH^{\times} \colon a \in H\}$ the associated reduced monoid of $H$. An element $a \in H$ is said to be {\it irreducible} (or an {\it atom}) if $a \notin H^{\times}$ and $a = bc$, with $b , c \in H$,  implies that $b \in H^{\times}$ or $c \in H^{\times}$. We denote by $\mathcal A (H)$ the set of atoms of $H$, by $\mathsf Z (H) = \mathcal F \big( \mathcal A (H_{\red}) \big)$  the {\it factorization monoid} of $H$,  and by $\pi \colon \mathsf Z (H) \to H_{\red}$ the {\it factorization homomorphism}, defined by $\pi (u) = u$ for all $u \in \mathcal A (H_{\red})$. For an element $a \in H$ and $k \in \N_0$, let
\begin{itemize}
\item[(i)] $\mathsf Z_H (a) = \mathsf Z (a) = \pi^{-1} (a) \subset \mathsf Z (H)$ be the {\it set of factorizations} of $a$,

\item[(ii)] $\mathsf Z_{H,k} (a) = \mathsf Z_k (a) = \{ z \in \mathsf Z (a) \colon |z|=k\}$ be the set of factorizations of $a$, which have length $k$, and

\item[(iii)] $\mathsf L_H (a) = \mathsf L (a) = \{|z| \colon z \in \mathsf Z (a) \} \subset \N_0$ be the {\it set of lengths} of $a$.
\end{itemize}
Thus, by definition, $\mathsf L (a) = \{1\}$ if and only if $a$ is an atom and $\mathsf L (a) = \{0\}$ if and only if $a \in H^{\times}$. We say that $H$ is
\begin{itemize}
\item[(i)] {\it atomic} if $\mathsf Z (a) \ne \emptyset$ for all $a \in H$ (equivalently, every $a \in H \setminus H^{\times}$ can be written as a finite product of atoms),

\item[(ii)] an \FF-{\it monoid} (finite-factorization monoid) if $\mathsf Z (a)$ is finite and nonempty for all $a \in H$, and

\item[(iii)] a \BF-{\it monoid} (bounded-factorization monoid) if $\mathsf L (a)$ is finite and nonempty for all $a \in H$.
\end{itemize}
Let $H$ be a \BF-monoid. We denote by
\[
\mathcal L (H) = \{\mathsf L (a) \colon a \in H \}
\]
the {\it system of sets of lengths} of $H$.
A submonoid $S \subset H$ is called {\it divisor-closed} if $a \in H$ and $b \in S$ with $a \t b$ implies that $a \in S$. If $a \in S$, then
\[
\mathsf Z_S (a) = \mathsf Z_H (a) \quad \text{and} \quad \mathsf L_S (a) = \mathsf L_H (a)  \,.
\]

\smallskip
Let $G$ be an additive abelian group and let $G_0 \subset G$ be a subset. For every $m \in \N$, we set $mG = \{mg \colon g \in G \}$. The elements of $\mathcal F (G_0)$ are called {\it sequences} over $G_0$. Let $S\in \mathcal F (G_0)$. Then we write
\[
S = g_1 \cdot \ldots \cdot g_{\ell} = \prod_{g \in G_0} g^{\mathsf v_g (S)} \,,
\]
and in this notation we tacitly assume that $\ell \in \N_0$ and $g_1, \ldots , g_{\ell} \in G_0$. If $\varphi \colon G \to G'$ is a group homomorphism, then $\varphi (S) = \varphi (g_1) \cdot \ldots \cdot \varphi (g_{\ell})$,  $-S = (-g_1) \cdot \ldots \cdot (-g_{\ell})$,
\[
|S|=\ell = \sum_{g \in G_0} \mathsf v_g (S) \in \N_0
\]
is the {\it length} of $S$, and
\[
\sigma (S) = g_1 + \ldots + g_{\ell} \in G
\]
is the {\it sum} of $S$. We say that $S$ is
\begin{itemize}
\item[(i)] {\it squarefree} (in $\mathcal F (G_0)$) if $g^2 \nmid S$ for all $g \in G_0$,

\item[(ii)] a {\it zero-sum sequence} if $\sigma (S) = 0$, and

\item[(iii)] a {\it plus-minus weighted zero-sum sequence} if there are $\varepsilon_1, \ldots, \varepsilon_{\ell} \in \{-1, 1\}$ such that \\ {$\varepsilon_1 g_1 + \ldots + \varepsilon_{\ell} g_{\ell} = 0$.}
\end{itemize}
We denote by $\mathcal B (G_0)$ the set of zero-sum sequences over $G_0$ and by $\mathcal B_{\pm} (G_0)$ the set of plus-minus weighted zero-sum sequences over $G_0$. These are submonoids of $\mathcal F (G_0)$ with the obvious inclusions
\[
\mathcal B (G_0) \subset \mathcal B_{\pm} (G_0) \subset \mathcal F (G_0) \,.
\]
Whenever we speak of a squarefree sequence, we mean that the sequence is squarefree in $\mathcal F (G)$. If $G_1 \subset G_0$ is a subset, then $\mathcal B (G_1) \subset \mathcal B (G_0)$ and $\mathcal B_{\pm} (G_1) \subset \mathcal B_{\pm} (G_0)$ are divisor-closed submonoids.
If every non-zero element $g \in G_0$ has order two, then $\mathcal B (G_0) = \mathcal B_{\pm} (G_0)$. Both monoids, $\mathcal B (G_0)$ and $\mathcal B_{\pm} (G_0)$, are \FF-monoids and hence \BF-monoids. Since the inclusion $\mathcal B (G_0) \hookrightarrow \mathcal F (G_0)$ is a divisor homomorphism, $\mathcal B (G_0)$ is  a Krull monoid.

A monoid homomorphism $\theta \colon H \to B$ is said to be a {\it transfer homomorphism} if the following two conditions are satisfied.
\begin{enumerate}
\item[{\bf (T\,1)\,}] $B = \theta(H) B^\times$  and  $\theta^{-1} (B^\times) = H^\times$.

\item[{\bf (T\,2)\,}] If $u \in H$, \ $b,\,c \in B$  and  $\theta (u) = bc$, then there exist \ $v,\,w \in H$ such that  $u = vw$, \  $\theta (v) \in bB^{\times}$, and  $\theta (w) \in c B^{\times}$.
\end{enumerate}
A transfer homomorphism  allows to pull back various arithmetic properties from $B$ to $H$. In particular, we have $\mathcal L (H) = \mathcal L (B)$. The classic example of a transfer homomorphism stems from the theory of  Krull monoids. They allow a transfer homomorphism to  monoids of zero-sum sequences $\mathcal B (G_0)$, where $G_0$ is a subset of the class group of the Krull monoid (\cite{Ge-HK06a}). A monoid is said to be {\it transfer Krull} if it allows a transfer homomorphisms to a Krull monoid (equivalently, to a monoid of zero-sums sequences) (see \cite{Ge-Zh20a, Ba-Re22a} for examples of transfer Krull monoids that are not Krull). If $G$ is an abelian group, then for the monoid of plus-minus weighted zero-sum sequences, we have the following characterizations (\cite[Corollary 3.5]{F-G-R-Z24a}). There are equivalent.
\begin{itemize}
\item[(a)]  $\mathcal B_{\pm} (G)$ is Krull.

\item[(b)] $\mathcal B_{\pm} (G)$ is transfer Krull.

\item[(c)]  $G$ is an elementary $2$-group.
\end{itemize}
On the other hand, in \cite[Theorems 3.2 and 3.5]{Ge-HK-Zh22}, it was proved that  norm monoids of Galois number fields and monoids of positive integers, which can be represented by certain binary quadratic forms, allow transfer homomorphisms to monoids of plus-minus weighted zero-sum sequences. More on this direction can be found in \cite{HK14b, B-M-O-S22, Co-Ha25a}.

We end this section with a simple lemma, dealing with the relationship of atoms in $\mathcal B (G)$ and in $\mathcal B_{\pm} (G)$. Clearly, we have
\begin{equation} \label{atom-inclusion}
\mathcal A \big( \mathcal B_{\pm} (G) \big) \cap \mathcal B (G) \subset \mathcal A \big( \mathcal B (G) \big) \,.
\end{equation}
The next lemma guarantees the reverse implication for groups with trivial $2$-torsion (for finite groups see \cite[Theorem 6.1]{B-M-O-S22}).

\smallskip
\begin{lemma} \label{2.1}
Let $G$ be an abelian group and suppose that there is no $g \in G$ with $\ord (g)=2$. Then $\mathcal A \big( \mathcal B (G) \big) = \mathcal A \big( \mathcal B_{\pm} (G) \big) \cap \mathcal B (G)$. In particular, we have $\mathsf L_{\mathcal B (G)} (B)  \subset \mathsf L_{\mathcal B_{\pm} (G)} (B)$ for all $B \in \mathcal B (G)$.
\end{lemma}

\begin{proof}
Clearly, it suffice to show that statement on the atoms. Then the 'in particular' on sets of lengths follows immediately. Let $A \in \mathcal A \big( \mathcal B (G) \big)$ and suppose that there are $B_1, B_2 \in \mathcal B_{\pm} (G)$ with $A = B_1B_2$ and with $B_1 \ne 1$. We assert that $B_1=A$.

Since $B_1, B_2 \in \mathcal B_{\pm} (G)$, there are $B_1^+, B_1^-, B_2^+, B_2^- \in \mathcal F (G)$ such that
\[
B_1 = B_1^+B_1^-, \ B_2 = B_2^+B_2^-, \ \sigma (B_1^+)=\sigma (B_1^-), \ \text{and} \ \sigma (B_2^+) = \sigma (B_2^-) \,.
\]
By symmetry, we may suppose without restriction that $B_1^+ \ne 1$. Then it follows that
\[
0 = \sigma (A) = \sigma (B_1^+)+\sigma (B_1^-)+\sigma (B_2^+)+\sigma (B_2^-) = 2 \big(  \sigma (B_1^+)+ \sigma (B_2^+) \big) \,.
\]
Since, by assumption, $G$ has no elements of order two, we infer that $\sigma (B_1^+)+ \sigma (B_2^+)=0$, whence $1 \ne B_1^+B_2^+ \in \mathcal B (G)$. Since $A \in \mathcal A \big( \mathcal B (G) \big)$, it follows that $A = B_1^+B_2^+$, whence $B_1^-=B_2^-=1$ and so $B_1=A$.

This implies that $\mathcal A \big( \mathcal B (G) \big) \subset \mathcal A \big( \mathcal B_{\pm} (G) \big)$, and hence the asserted equality holds by Relation \eqref{atom-inclusion}.
\end{proof}

\smallskip
\section{On the construction of a group having a given $L$ with $\min L = 2$ as a set of lengths} \label{3}
\smallskip

The goal in this section is to prove the following proposition.

\smallskip
\begin{proposition} \label{3.1}
Let $R$ be a commutative ring, let $L \subset \N_{\ge 2}$ be a finite
 subset with $2 \in L$,  and let $f \colon L \to \N$  be a map with $\sum_{k \in L} f (k) \ge 3$. Then there exist
a finitely generated free $R$-module $G$ and a sequence $B \in \mathcal B (G)$ with the following properties.
\begin{enumerate}
\item[(a)] $B$ is squarefree in $\mathcal F (G)$.

\item[(b)] $\mathsf L_{\mathcal B_{\pm} (G)} (B) = \mathsf L_{\mathcal B (G)} (B) = L$.

\item[(c)] $|\mathsf Z_{\mathcal B_{\pm} (G),k} (B)| \ge |\mathsf Z_{\mathcal B (G),k} (B)| \ge f (k)$ for all $k \in L$. Moreover, both inequalities are equalities for  $k > \min L$.
\end{enumerate}
\end{proposition}

\smallskip
We proceed in three subsections. First, we put together the setting in which $G$ and $B$ are defined. Then we study the atoms in $\mathcal B (G)$ and in $\mathcal B_{\pm} (G)$. In the third subsection, we  give the proof of Proposition \ref{3.1}

\smallskip
\subsection{On the construction of the $R$-module $G$ and the wanted sequence $B \in \mathcal B (G)$}~

As said in the introduction, the $R$-module $G$ and the sequence $B$ are the same as in \cite{Ka99a}. We carefully provide all definitions and gather all the required properties. The proofs of these properties can be found in \cite{Ka99a} and in \cite[Chapter 7.4]{Ge-HK06a}.

Let $R$ be a commutative ring, let $L \subset \N_{\ge 2}$ be a finite
 subset with $2 \in L$,  and let $f \colon L \to \N$  be a map with $s := \sum_{k \in L} f (k) \ge 3$.
We  consider a tuple of finite sets $(X_1, \ldots, X_s)$  such that
\begin{equation} \label{setting-1}
L = \{|X_1|, \ldots, |X_s|\} \quad \text{and} \quad f(k) = \bigl|
\bigl\{ i \in [1,s] \colon |X_i| = k \bigr\} \bigr|
\quad\text{for every}\quad k \in L \,.
\end{equation}
We set
\[
X = \prod^s_{j=1} X_j \quad\text{and}\quad X_J = \prod_{j \in J}X_j \quad\text{for}\quad
\emptyset \ne J \subset [1,s] \,.
\]
For $\emptyset \ne J \subset [1,s]$, let $p_J \colon X \to X_J$ be the canonical projection and, for $j \in [1,s]$, let $p_j = p_{\{j\}}\colon X \to X_j$. For $z = (z_1, \ldots, z_s) \in X$, $i \in [1,s]$, and  a nonempty subset $J \subset [1,s]$, we define
\[
X^{(z)}_i = X_i \setminus \{z_i\} \quad\text{and}\quad X^{(z)}_J = \prod_{j\in J} X^{(z)}_j \,.
\]
Furthermore, for two subsets $M, N \subset X$, let
\[
\Delta (M,N) = (M \setminus N) \cup (N \setminus M)
\]
denote the symmetric difference of $M$ and $N$. For two elements $x = (x_1, \ldots, x_s), \,y = (y_1, \ldots, y_s) \in X$,
we set  $\Delta (x,y) = \{i \in [1,s] \colon x_i \ne y_i\}$.

Now let $R^X$ be the $R$-algebra of all maps $X \to R$, with pointwise addition, pointwise multiplication, and with scalar multiplication by elements from $R$. For a subset $M \subset X$, let $\chi_M \in R^X$ be the characteristic function of $M$, defined by
\[
\chi_M (x) =
\begin{cases}
1\,, & \text{ if } \ x \in M \,,\\
0\,, & \text{ if } \ x \notin M \,.
\end{cases}
\]
For an element $x \in X$,  we set $\chi_x = \chi_{\{x\}} \in R^X$, and $\boldsymbol 1 = \chi_X$ is the constant function with value $1 \in R$.

Next, we define the wanted finitely generated free $R$-module. We set
\begin{equation} \label{setting-2}
V = \bigl \langle \{ \chi_{p^{-1}_i (y)} \colon i \in [1,s]\,, \ y \in X_i \} \bigr\rangle_R \,, \quad G = R^X \negthinspace/V\,,
\end{equation}
and, for every $z \in X$,
\[
W_z = \bigl \langle \bigl\{ \chi_{p^{-1}_J (y)} \colon J
\subset [1,s]\,, \ |J| \ge 2\,, \ y \in X^{(z)}_J \bigr\}
\bigr\rangle_R \,.
\]
We have that  $V = \bigl\langle \{\boldsymbol {1}\} \cup \{ \chi_{p^{-1}_i
(y)} \colon i \in [1,s]\,, \ y \in X^{(z)}_i \}\bigr \rangle_R$,
$R^X = V \oplus W_z$, and so $G$ is a finitely generated free $R$-module. We denote by $P_z \colon R^X = V \oplus W_z \to V$ the associated projection. If $M \subset X \setminus \{z\}$, then 
\begin{equation} \label{$P_z-1$}
P_z (\chi_z) =\boldsymbol {1} - \sum^s_{i=1} \sum_{y\in X^{(z)}_i} \chi_{p^{-1}_i(y)} \,, \quad P_z (\chi_x) = \begin{cases}
\chi_{p^{-1}_i (x_i)}\,, &\text{if } \ \Delta(z,x) = \{i\} \,,\\
\quad \boldsymbol 0\,, & \text{if } \ |\Delta (z,x)| \ge 2 \,,
\end{cases}
\end{equation}
and
\begin{equation} \label{$P_z-2$}
P_z (\chi_M) = \sum^s_{i=1} \chi_{p^{-1}_i (Y_i)} \,,
\end{equation}
where, for each $i \in [1,s]$,  $Y_i$ is a  subset of $X^{(z)}_i$.

\smallskip
For a subset $M \subset X$,  we set
\begin{equation} \label{setting-3}
B_M = \prod_{x\in M} (\chi_x + V) \in \mathcal F (G) \quad \text{and} \quad B:= B_X
\end{equation}
Since  $\chi_x - \chi_z \notin V$ for any two distinct elements of $X$,  all sequences $B_M$ are
squarefree. If $M,\, M',\, M''$ are subsets of $X$, then $B_M =
B_{M'} B_{M''}$ if and only if $M = M' \uplus M''$. Since
\[
\sigma (B_M) = \sum_{x \in M} (\chi_x + V) = \chi_M + V \,,
\]
we have $B_M \in \mathcal B(G)$ if and only if $\chi_M \in V$. In
particular, $\boldsymbol 1 = \chi_X \in V$ implies $B \in \mathcal
B (G) \subset \mathcal B_{\pm} (G)$.  Moreover, for every divisor $T$ of $B$ in $\mathcal F (G)$ there is a unique subset  $M \subset X$ with $T = B_M$.

\smallskip
\subsection{On the atoms of $\mathcal B (G)$ and $\mathcal B_{\pm} (G)$}

\smallskip
\begin{lemma} \label{3.4}~

\begin{enumerate}
\item If $i \in [1,s]$,  $Y_i \subset X_i$ and $M = p^{-1}_i (Y_i)$, then  $B_M \in \mathcal B (G)$, and  $B_M \in \mathcal A \big( \mathcal B (G) \big)$  if and only if  $|Y_i| = 1$.

\item If  $\text{\rm char}(R) \ne 2$,  $M \subset X$, and $B_M \in \mathcal B (G)$, then $M = p^{-1}_i (Y_i)$ for some $i \in [1,s]$ and some $Y_i \subset X_i$.

\item     Let  $\text{char} (R) = 2$,  $M \subsetneq X$ and
$B_M \in \mathcal B (G)$. Then, for every $i \in [1,s]$, there is a set $Y_i \subset X_i$ such that
\[\qquad \quad
M = \bigl\{ (x_1, \ldots, x_s) \in X \colon |\{i \in [1,s]\colon x_i \in Y_i\}| \quad\text{is odd} \,\bigr\} \,.
\]
If $B_M \notin \mathcal A \big( \mathcal B(G) \big)$, then $M = p^{-1}_i (Y_i)$ for some $i \in [1,s]$ and some subset $Y_i \subset X_i$. Moreover, if $M' \subset X$ is such that $B_{M'} \in \mathcal B (G)$ and $B_{M'} \t B_M$, then $M' = p^{-1}_i (Y'_i)$ for some subset $Y'_i \subset Y_i$.
\end{enumerate}
\end{lemma}

\begin{proof}
1. If $i \in [1,s]$,  $ Y_i \subset X_i$ and $M = p^{-1}_i (Y_i)$, then
\[
\chi_M = \sum_{y \in Y_i} \chi_{p^{-1}_i (y)} \in V
\]
and thus $B_M \in \mathcal B (G)$. Clearly, $B_{\emptyset} = 1 \in
\mathcal B (G)$. If $i,\,j \in [1,s]$, \ $\emptyset \ne Y_i
\subset X_i$ and $\emptyset \ne Y'_j \subset X_j$, then $p^{-1}_i
(Y_i) \subset p^{-1}_j (Y'_j)$ if and only if $i = j$ and $Y_i
\subset Y'_j$. Hence $B_M \in \mathcal A \big( \mathcal B (G) \big)$ if and only if
$|Y_i| = 1$.

2. Clearly, we have  $X = p^{-1}_1 (X_1)$. Now suppose
that $M \subsetneq X$ and $\chi_M \in V$. If $z \in X
\setminus M$, then by \eqref{$P_z-2$},
\[
\chi_M = P_z (\chi_M) = \sum^s_{i=1} \chi_{p^{-1}_i (Y_i)} \,,
\]
where $Y_i \subset X_i$ for each $i \in [1,s]$. Since
\[
\chi_M = \chi^2_M = \sum^s_{i=1} \chi_{p^{-1}_i (Y_i)} + 2 \sum^s_{\substack{i,j=1 \\ i<j}} \chi_{p^{-1}_i (Y_i)} \chi_{p^{-1}_j (Y_j)} \,,
\]
it follows that
\[
0 = 2 \sum^s_{\substack{i,j=1 \\ i<j}} \chi_{p^{-1}_i (Y_i)} \chi_{p^{-1}_j (Y_j)} = 2 \sum^s_{\substack{i,j=1 \\ i<j}}
\sum_{y \in Y_i \times Y_j} \chi_{p^{-1}_{\{i,j\}}} (y) \,.
\]
By \eqref{$P_z-1$}, this can happen only if $Y_i \times Y_j =
\emptyset$ for all distinct $i, j \in [1,s]$. Hence there is
at most one $i \in [1,s]$ with $Y_i \ne \emptyset$.

3. Let $z \in X \setminus M$. If $B_M
\in \mathcal B (G)$, then \eqref{$P_z-2$} implies
\[
\chi_M = P_z (\chi_M) = \sum^s_{i=1} \chi_{p^{-1}_i (Y_i)} \,,
\]
where $Y_i \subset X^{(z)}_i$ for all $i \in [1,s]$. For $x =(x_1, \ldots, x_s) \in X$ we have
\[
\chi_M (x) = \sum^s_{i=1} \chi_{p^{-1}_i (Y_i)} (x) \ = \ \bigl|
\{ i \in [1,s] \colon x_i \in Y_i \} \bigr|\, 1_R \,,
\]
and therefore \ $M = \bigl\{ (x_1, \ldots, x_s) \in X  \colon |\{ i \in [1,s] \colon x_i \in Y_i\}| \quad\text{is odd} \bigr\}$.

\smallskip
If $B_M \notin \mathcal A \big( \mathcal B (G) \big)$, then there is a set $\emptyset \ne M' \subsetneq M$ such that $B_{M'} \in \mathcal B (G)$. For every $i \in [1,s]$, let $Y'_i \subset X^{(z)}_i$ be such that
\[
M' = \bigl\{ (x_1, \ldots, x_s) \in X  \,\colon  \,|\{ i\in [1,s] \colon x_i \in Y'_i \}| \quad\text{is odd}\, \bigr\} \,.
\]
If $j \in [1,s]$ and $y \in Y'_j$, then \ $(z_1, \ldots, z_{j-1},
y, z_{j+1}, \ldots, z_s) \in M' \subset M$, and thus $y \in Y_j$.
Hence we obtain $Y'_j \subset Y_j$ for all $j \in [1,s]$, and $M'
\subsetneq M$ implies $Y'_i \subsetneq Y_i$ for some $i \in
[1,s]$, say $Y'_1 \subsetneq Y_1$. We assert that $Y_j = Y'_j =
\emptyset$ for all $j \in [2,s]$. It suffices to do the proof for
$j = 2$. Assume the contrary. If $Y'_2 \ne \emptyset$, let $y_1
\in Y_1 \setminus Y'_1$ and $y_2 \in Y'_2$. Then $(y_1, y_2, z_3,
\ldots, z_s) \in M' \setminus M$, a contradiction. Hence $Y'_2 =
\emptyset$. If $Y_2 \ne \emptyset$, let $y_1 \in Y'_1$ and $y_2
\in Y_2$. Then $(y_1, y_2, z_3, \ldots, z_s) \in M' \setminus M$,
again a contradiction.
\end{proof}

\smallskip
\begin{lemma} \label{3.5}
Suppose that $\text{\rm char} (R) \ne 2$.  For a subset $M \subset X$, the following conditions are equivalent.
\begin{enumerate}
\item[(a)] $B_M \in \mathcal B_{\pm} (G) \setminus \mathcal B (G)$ and  $B_M$ divides $B$ in $\mathcal B_{\pm} (G)$.

\item[(b)] There are $i, j \in [1,s]$ distinct and $\emptyset \ne Y_i \subsetneq X_i$, $\emptyset \ne Y_j \subsetneq X_j$ such that $M = \Delta (p_i^{-1} (Y_i), p_j^{-1} (Y_j))$.
\end{enumerate}
\end{lemma}

\begin{proof}
(b) $\Longrightarrow$ (a) We have
\[
M = \Delta (p_i^{-1} (Y_i), p_j^{-1} (Y_j)) = \underbrace{\big(p_i^{-1} (Y_i) \setminus p_j^{-1} (Y_j) \big)}_{M^+} \cup \underbrace{\big( p_j^{-1} (Y_j) \setminus p_i^{-1} (Y_i) \big)}_{M^-} \,.
\]
Then $M = M^+ \cup M^+$ and $\chi_{M^+} - \chi_{M^-} = \chi_{p_i^{-1} (Y_i)} - \chi_{p_j^{-1} (Y_j)} \in V$, whence $\prod_{x \in M} p (\chi_x) \in \mathcal B_{\pm} (G)$. The complement $X \setminus M$ has the same form, whence it also lies in $\mathcal B_{\pm} (G)$ and so $B_M$ is a divisor of $B$ in $\mathcal B_{\pm} (G)$.

Assume to the contrary that $\prod_{x \in M} p (\chi_x) \in \mathcal B (G)$. Then $\chi_{M^+} + \chi_{M^-} = 0 = \chi_{M^+} - \chi_{M^-}$, whence $2 \chi_{M^+} = 2 \chi_{M^-} = 0$. Since $\textrm{char} (R) \ne 2$, it follows that $M^+ = M^- = \emptyset$ and so $M = \emptyset$. On the other hand, $\emptyset \ne Y_i \subsetneq X_i$ and $\emptyset \ne Y_j \subsetneq X_j$ implies that $M = \Delta (p_i^{-1} (Y_i), p_j^{-1} (Y_j)) \ne \emptyset$, a contradiction.

\medskip
(a) $\Longrightarrow$ (b) We distinguish two cases (these cases have a large overlap).

\smallskip
\noindent
CASE 1: $\text{\rm char} (R) \ne 3$.

Let $B_M \in \mathcal B_{\pm} (G) \setminus \mathcal B (G)$.
Since $\prod_{x \in M} p ( \chi_x ) \in \mathcal B_{\pm} (G)$, there are $M^+, M^- \subset M$ such that $M^+ \uplus M^- = M$ and $\chi_{M^+} - \chi_{M^-} \in V$. Since $\prod_{x \in M} p ( \chi_x ) \notin \mathcal B (G)$, we have $M^+ \ne \emptyset$, $M^- \ne \emptyset$, and $M \ne X$.

Let $z \in X \setminus M$ be given. By \eqref{$P_z-1$}, we have, for every $x \in M$, that
\[
\chi_x \equiv \begin{cases}
             0 & |\Delta (x,z)|\ge 2 \quad \mod W_z \\
              \chi_{p_i^{-1} (x_i)} & \Delta (x, z) = \{i\} \quad \mod W_z
              \end{cases}
\]
Therefore, it follows that
\[
\begin{aligned}
\chi_{M^+} - \chi_{M^-} & = \sum_{x \in M^+} \chi_x - \sum_{x \in M^-} \chi_x \equiv \sum_{x \in M^+, |\Delta (x,z)|=1} \chi_x - \sum_{x \in M^-, |\Delta (x,z)|=1} \chi_x \\
 & = \sum_{i=1}^s \sum_{x \in M^+, \Delta (x,z) = \{i\}} \chi_{p_i^{-1}(x_i)} \ -  \
     \sum_{i=1}^s \sum_{x \in M^-, \Delta (x,z) = \{i\}} \chi_{p_i^{-1}(x_i)} \quad \mod W_z \,.
\end{aligned}
\]
For every $i \in [1,s]$, we define
\[
\begin{aligned}
Y_i^+ & = \{ y_i \in X_i \colon (z_1, \ldots, z_{i-1},y_i, z_{i+1}, \ldots, z_n) \in M^+ \} \quad \text{and} \\
Y_i^- & = \{ y_i \in X_i \colon (z_1, \ldots, z_{i-1},y_i, z_{i+1}, \ldots, z_n) \in M^- \}
\end{aligned}
\]
and obtain that
\[
\chi_{M^+} - \chi_{M^-} \equiv \sum_{i=1}^s \chi_{p_i^{-1} (Y_i^+)} - \sum_{i=1}^s \chi_{p_i^{-1} (Y_i^-)} \ \mod W_z \,.
\]
Since $\chi_{M^+} - \chi_{M^-}$, $\sum_{i=1}^s \chi_{p_i^{-1} (Y_i^+)}$, and $\sum_{i=1}^s \chi_{p_i^{-1} (Y_i^-)}$ are in $V$, it follows that
\[
\chi_{M^+} - \chi_{M^-} = \sum_{i=1}^s \chi_{p_i^{-1} (Y_i^+)} - \sum_{i=1}^s \chi_{p_i^{-1} (Y_i^-)}  \,.
\]
Assume to the contrary that there are distinct $i, j \in [1,s]$ such that $Y_i^+ \ne \emptyset$ and $Y_j^+ \ne \emptyset$. We pick $u \in X$ with $u_i \in Y_i^+$, $u_j \in Y_j^+$, and $u_k = z_k$ for all $k \in [1,s] \setminus \{i,j\}$. Since $z \notin M$, we have $z_k \notin Y_k^+ \cup Y_k^-$ and therefore
\[
u \notin p_k^{-1} (Y_k^+) \cup p_k^{-1} (Y_k^-) \quad \text{for any $k \in [1,s] \setminus \{i,j\}$} \,.
\]
Since $M^+ \cap M^- = \emptyset$, we obtain that $Y_k^+ \cap Y_k^- = \emptyset$ for all $k \in [1,s]$. This implies that $u \notin p_k^{-1} (Y_k^-)$ for all $k \in [1,s]$, but $u \in p_i^{-1} (Y_i^+)$ and $u \in p_j^{-1} (Y_j^+)$. Therefore, we obtain that
\[
2 = \sum_{k=1}^s \chi_{p_k^{-1} (Y_k^+)} (u) - \sum_{k=1}^s \chi_{p_k^{-1} (Y_k^-)} (u) = \chi_{M^+} (u) - \chi_{M^-} (u) \in \{-1, 0, 1\} \,,
\]
a contradiction to $\text{\rm char} (R) \ne 3$.

The same arguments yield a contradiction, if we assume that there are distinct $i, j \in [1,s]$ such that $Y_i^- \ne \emptyset$ and $Y_j^- \ne \emptyset$.

Therefore, there are $i, j \in [1,s]$ with
\[
\chi_{M^+} - \chi_{M^-} = \chi_{p_i^{-1} (Y_i^+)} - \chi_{p_j^{-1} (Y_j^-)} =
\chi_{p_i^{-1} (Y_i^+) \setminus p_j^{-1} (Y_j^-)} - \chi_{p_j^{-1} (Y_j^-) \setminus p_i^{-1} (Y_i^+)} \,.
\]
This implies that
\[
M^+ = p_i^{-1} (Y_i^+) \setminus p_j^{-1} (Y_j^-) \quad \text{and} \quad M^- = p_j^{-1} (Y_j^-) \setminus p_i^{-1} (Y_i^+) \,,
\]
whence
\[
M = \Delta \big( p_i^{-1} (Y_i^+), p_j^{-1} (Y_j^-) \big) \,.
\]
Assume to the contrary that $i=j$. Then we obtain
\[
M = \Delta \big( p_i^{-1} (Y_i^+), p_i^{-1} (Y_i^-) \big) = p_i^{-1} ( \Delta (Y_i^+, Y_i^-) ) \,,
\]
which implies that $\chi_M \in V$ and hence $\prod_{x \in M} p (\chi_x) \in \mathcal B (G)$, a contradiction.

Assume to the contrary that $Y_i^+ = \emptyset $ or that $Y_i^+ = X_i$. This implies that
\[
M = \Delta \big( p_i^{-1} (Y_i^+), p_j^{-1} (Y_j^-) \big) = p_j^{-1} (Y_j^-) \,,
\]
which yields $\prod_{x \in M} p (\chi_x) \in \mathcal B (G)$, a contradiction. The same argument shows that $\emptyset \ne Y_j^- \ne X_j$.

\smallskip
\noindent
CASE 2: $\textrm{char} (R)$ is either zero or odd.

Since $B_M$ divides $B$, there are pairwise disjoint sets $M^+, M^-, N^+$, $N^- \subset X$ with
\[
M^+ \uplus M^- \uplus N^+ \uplus N^- = X
\]
and with $\sigma (B_{M^+}) = \sigma (B_{M^-})$ and $\sigma (B_{N^+}) = \sigma (B_{N^-})$. This implies that
\[
0=\sigma (B) = \sigma (B_{M^+ \uplus M^- \uplus N^+ \uplus N^-}) = \sigma (B_{M^+}) + \sigma (B_{M^-})+\sigma (B_{N^+}) + \sigma (B_{N^-}) = 2 \big( \sigma (B_{M^+}) + \sigma (B_{N^+}) \big) \,.
\]
Since $\textrm{char} (R)$ is either zero or odd, it follows that $\sigma (B_{M^+}) + \sigma (B_{N^+}) = 0$. Thus, $B_{M^+}B_{N^+} \in \mathcal B (G)$, whence $M^+\cup N^+ = p_i^{-1}(Y_i)$ for some $i \in [1,s]$ and some $Y_i \subset X_i$. For its complement, we get $M \setminus (M^+ \cup N^+) = M^- \cup N^- = p_i^{-1} (X_i \setminus Y_i)$. Similarly, we get, for some $j \in [1,s]$ and some $W_j \subset X_j$,
\[
M^+ \cup N^- = p_j^{-1} (X_j \setminus W_j) \quad \text{and} \quad M^- \cup N^+ = p_j^{-1} (W_j) \,.
\]
Therefore, $M^- = (M^- \cup N^-) \cap (M^- \cup N^+) = p_i^{-1} (X_i \setminus Y_i) \cap p_j^{-1} (W_j)$ and
\[
M=M^+ \cup M^- = \Delta \big( p_i^{-1}(Y_i), p_j^{-1} (W_j) \big) \,. \qedhere
\]
\end{proof}

\smallskip
\begin{lemma} \label{3.6}
Let $i, j \in [1,s]$ distinct, $\emptyset \ne Y_i \subsetneq X_i$, $\emptyset \ne Y_j \subsetneq X_j$ and $M = \Delta (p_i^{-1} (Y_i), p_j^{-1} (Y_j))$. Then $B_M \in \mathcal A \big( \mathcal B_{\pm} (G) \big)$.
\end{lemma}

\begin{proof}
Let $N \subset M$ be a nonempty subset with $B_N \in \mathcal B_{\pm} (G)$. We have to show that $N=M$ and to do so, we distinguish two cases.

\smallskip
\noindent
CASE 1: $B_N \in \mathcal B (G)$.

By Lemma \ref{3.4}, there are $k \in [1,s]$ and $W_k \subset X_k$ such that $N = p_k^{-1} (W_k)$. Since $\emptyset \ne N \ne X$, we obtain that $\emptyset \ne W_k \subsetneq X_k$, whence
\[
p_k^{-1} (W_k) \subset \Delta (p_i^{-1} (Y_i), p_j^{-1} (Y_j) ) \,.
\]
\smallskip
\noindent
CASE 1.1: $k \in \{i,j\}$, say $k=i$.

Let $x_i \in W_i$. We choose $y_j \in X_j \setminus Y_j$ and $u \in X$ with $u_i=w_i$ and $u_j =y_j$. Then we have $u \in p_k^{-1} (W_k) \subset \Delta (p_i^{-1} (Y_i), p_j^{-1} (Y_j) )$. Since $y_j \notin Y_j$, it follows that $w_i=u_i \in Y_i$, whence $W_i \subset Y_i$. Now we pick some $u \in X$ with $u_i \in W_i \subset Y_i$ and $u_j \in Y_j$. This implies that
\[
u \in p_i^{-1} (W_i) \quad \text{but} \quad u \notin \Delta (p_i^{-1} (Y_i), p_j^{-1} (Y_j) ) \,,
\]
a contradiction.

\smallskip
\noindent
CASE 1.2: $k \notin \{i,j\}$.

We pick some $u \in X$ with $u_k \in W_k$, $u_i \in Y_i$, and $u_j \in Y_j$. Then
\[
u \in p_k^{-1} (W_k) \quad \text{but} \quad u \notin \Delta (p_i^{-1} (Y_i), p_j^{-1} (Y_j) ) \,,
\]
a contradiction.

\smallskip
\noindent
CASE 2:  $B_N \notin \mathcal B (G)$.

By Lemma \ref{3.4}, there are $k, \ell \in [1,s]$ distinct, $\emptyset \ne W_k \subsetneq X_k$, and $\emptyset \ne W_{\ell} \subsetneq X_{\ell}$ such that $N = \Delta (p_k^{-1} (W_k), p_{\ell}^{-1} (W_{\ell}))$, and we obtain that
\[
\Delta (p_k^{-1} (W_k), p_{\ell}^{-1} (W_{\ell})) \subset \Delta (p_i^{-1} (Y_i), p_j^{-1} (Y_j)) \,.
\]
We distinguish three cases.

\smallskip
\noindent
CASE 2.1: $\{k, \ell\} \cap \{i,j \} = \emptyset$.

We pick $u \in X$ with $u_k \in W_k$, $u_{\ell} \notin W_{\ell}$, $u_i \in Y_i$, and $u_j \in Y_j$. Then
\[
u \in \Delta (p_k^{-1} (W_k), p_{\ell}^{-1} (W_{\ell})) \quad \text{but} \quad u \notin \Delta (p_i^{-1} (Y_i), p_j^{-1} (Y_j)) \,,
\]
a contradiction.

\smallskip
\noindent
CASE 2.2:  $|\{k, \ell\} \cap \{i,j \}| = 1$, say $k=i$.

Let $w_i \in W_i$. We choose $u \in X$ with $u_i =w_i$, $u_{\ell} \notin W_{\ell}$, and $u_j \notin Y_j$. Then
\[
u \in \Delta (p_i^{-1} (W_i), p_{\ell}^{-1} (W_{\ell})) \quad \text{and} \quad u \in \Delta (p_i^{-1} (Y_i), p_j^{-1} (Y_j)) \,.
\]
Since $u_j \notin Y_j$, it follows that $w_i=u_i \in Y_i$, and hence $W_i \subset Y_i$.

Now we choose $u \in X$ with $u_i \in W_i$, $u_{\ell} \notin W_{\ell}$, and $u_j \in Y_j$. Then
\[
u \in \Delta (p_i^{-1} (W_i), p_{\ell}^{-1} (W_{\ell})) \quad \text{but} \quad u \notin \Delta (p_i^{-1} (Y_i), p_{\ell}^{-1} (W_{\ell})) \,,
\]
a contradiction. 

\smallskip
\noindent
CASE 2.3:  $\{k, \ell\} = \{i,j \}$, say $k=i$ and $\ell = j$.

Then
\[
\Delta \big( p_i^{-1} (W_i), p_j^{-1} (W_j) \big) \subset \Delta \big( p_i^{-1} (Y_i), p_j^{-1} (Y_j) \big) \,.
\]
Since $W_i$ is nonempty, we have $W_i \not\subset Y_i$ or $W_i \not\subset X_i \setminus Y_i$. Since
\[
\Delta \big( p_i^{-1} (Y_i), p_j^{-1} (Y_j) \big) = \Delta \big( p_i^{-1} (X_i \setminus Y_i), p_j^{-1} (X_j \setminus Y_j) \big) \,,
\]
we may assume that $W_i \not\subset  X_i \setminus Y_i$. We choose
\[
x_j \in X_j \setminus W_j, x_i \in W_i \cap Y_i, \ \text{and} \ u \in X \ \text{with} \ u_i=x_i \ \text{and} \ u_j = x_j \,.
\]
Then $u \in \Delta ( p_i^{-1} (W_i), p_j^{-1} (W_j))$ and hence $u \in \Delta ( p_i^{-1} (Y_i), p_j^{-1} (Y_j))$. Since $u_i=x_i \in Y_i$, we obtain that $x_j = u_j \in X_j \setminus Y_j$. Thus, it follows that $X_j \setminus W_j \subset X_j \setminus Y_j$ and so $Y_j \subset W_j$.

Now we choose some $w_i \in W_i$. Let $u \in X$ with $u_i=w_i$ and $u_j \in X_j \setminus W_j$. Thus,
\[
u \in \Delta ( p_i^{-1} (W_i), p_j^{-1} (W_j) ) \quad \text{and hence} \quad u \in \Delta (p_i^{-1} (Y_i), p_j^{-1} (Y_j)) \,,
\]
and, because $u_j \in X_j \setminus W_j$, we finally obtain that $w_i=u_i \in Y_i$. Summing up, we obtained that $W_i \subset Y_i$. Since $\emptyset \ne Y_j \subset W_j$, we have that $W_j \not\subset X_j \setminus Y_j$.

Interchanging $i$ and $j$ in the above arguments, we also obtain that $Y_i \subset W_i$ and $W_j \subset Y_j$. Thus, we obtain that $W_i = Y_i$ and $W_j = Y_j$, whence $M=N$.
\end{proof}

\smallskip
\begin{lemma} \label{3.7}
Let  $A \in \mathcal A \big( \mathcal B (G) \big)$ be a divisor of $B$ in $\mathcal B (G)$. Then $A \in \mathcal A \big( \mathcal B_{\pm} (G) \big)$.
\end{lemma}

\begin{proof}
By Lemma \ref{3.4}, there are $k \in [1,s]$ and $x_k \in X_k$ such that $A = \prod_{x \in p_k^{-1}(x_k)} p(\chi_x)$. Assume to the contrary that there is $C \in \mathcal B_{\pm} (G)$ with $1 \ne C \ne A$ such that $C \t A$ (in $\mathcal B_{\pm} (G)$). Since $A \in \mathcal A  \big( \mathcal B (G) \big)$, it follows that $C \notin \mathcal B (G)$. Thus, Lemma \ref{3.5} implies that there are $i, j \in [1,s]$ distinct and $\emptyset \ne Y_i \subsetneq X_i$, $\emptyset \ne Y_j \subsetneq X_j$ such that
\[
C = B_M \quad \text{with} \quad M = \Delta (p_i^{-1} (Y_i), p_j^{-1} (Y_j)) \,.
\]
This implies that
\[
\Delta (p_i^{-1} (Y_i), p_j^{-1} (Y_j)) \subset p_k^{-1} (x_k) \,.
\]
If $k \notin \{i, j\}$, then we choose $u \in X$ with $u_i \in Y_i$, $u_j \notin Y_j$, and $u_k \ne x_k$ (note that $|X_k| \ge 2$). This implies that
\[
u \in \Delta (p_i^{-1} (Y_i), p_j^{-1} (Y_j)) \quad \text{and} \quad u \notin p_k^{-1} (x_k) \,,
\]
a contradiction. Thus, we infer that $k \in \{i, j\}$, say $k=i$, and then
\[
\Delta (p_i^{-1} (Y_i), p_j^{-1} (Y_j)) \subset  p_i^{-1} (x_i) \,.
\]
Now we pick some $y_i \in Y_i$ and some $u \in X$ with $u_i = y_i$ and $u_j \notin Y_j$. Then $u \in p_i^{-1} (x_i)$, whence $y_i = u_i = x_i$ and thus $Y_i = \{x_i\}$. Now we choose some $u \in X$ with $u_i \ne x_i$ and $u_j \in Y_j$. Then
\[
u \in \Delta (p_i^{-1} (Y_i), p_j^{-1} (Y_j)) \quad \text{but} \quad u \notin p_i^{-1} (x_i) \,,
\]
a contradiction.
\end{proof}

\smallskip
\subsection{Proof of Proposition \ref{3.1}}

Let $R$ be a commutative ring, let $L \subset \N_{\ge 2}$ be a finite
 subset with $2 \in L$,  and let $f \colon L \to \N$  be a map with $s := \sum_{k \in L} f (k) \ge 3$. Furthermore, let all notation be as in \eqref{setting-1}. In particular, we have
\[
 L = \{|X_1|, \ldots, |X_s|\} \quad \text{and} \quad f(k) = \bigl|
\bigl\{ i \in [1,s] \colon |X_i| = k \bigr\} \bigr|
\quad\text{for every}\quad k \in L \,.
\]
and
\[
X = \prod^s_{j=1} X_j \quad\text{and}\quad X_J = \prod_{j \in J}X_j \quad\text{for}\quad
\emptyset \ne J \subset [1,s] \,.
\]
For every  $i \in [1,s]$, we have
\[
X = \biguplus_{y \in X_i} p_i^{-1}(y)\,, \quad \text{hence} \quad B = \prod_{y \in X_i} B_{p_i^{-1}(y)} \,.
\]
By Lemma \ref{3.4}, all $B_{p_i^{-1}(y)}$ are atoms of $\mathcal B (G)$ and by Lemma \ref{3.7} they are atoms of $\mathcal B_{\pm} (G)$. Thus,
\begin{equation} \label{factorizations-1}
Z_i = \prod_{y \in X_i} B_{p_i^{-1}(y)} \quad \text{is a factorization of $B$ in $\mathcal B (G)$}
\end{equation}
and
\[
Z_{\pm,i} = \prod_{y \in X_i} B_{p_i^{-1}(y)} \quad \text{is a factorization of $B$ in $\mathcal B_{\pm} (G)$} \,.
\]
This implies that
\[
\mathsf Z_{\mathcal B (G)} (B) \supset \{ Z_{1}, \ldots, Z_{s} \} \quad \text{and} \quad
\mathsf Z_{\mathcal B_{\pm} (G)} (B) \supset \{ Z_{\pm,1}, \ldots, Z_{\pm,s} \}
\]
We continue with the following assertion.

\smallskip
\noindent
{\bf A.} If there
exists some $Z\in \mathsf Z_{\mathcal B (G)} (B) \setminus \{Z_1, \ldots, Z_s\}$ or some
$Z' \in \mathsf Z_{\mathcal B_{\pm} (G)} (B) \setminus \{Z_{\pm, 1}, \ldots, Z_{\pm, s} \}$
then  $2 = |Z| = |Z'|$.

\smallskip
Suppose that {\bf A} holds. Since $2 \in L$ and since, for all $i \in [1,s]$, $|X_i| = |Z_i| = |Z_{\pm, i}|$, it follows that
\[
\mathsf L_{\mathcal B (G)} (B) = \{|X_1|, \ldots, |X_s|\} = L = \mathsf L_{\mathcal B_{\pm} (G)} (B)  \,.
\]
Furthermore, for all $k \in L \setminus \{2\}$, we have
\[
|\mathsf Z_{\mathcal B (G), k} (B)| = |\mathsf Z_{\mathcal B_{\pm} (G), k} (B)|  =   |\bigl\{ i \in [1,s]\, \bigm|\, |X_i|
= k \bigr\} \bigr| = f(k) \quad  \,.
\]
If   $k=2 = \min L$, then Lemma \ref{3.7} implies that $|\mathsf Z_{\mathcal B (G), k} (B)| \le |\mathsf Z_{\mathcal B_{\pm} (G), k} (B)|$ and {\bf A} implies that $|\mathsf Z_{\mathcal B (G), 2} (B)| \ge \bigl|
\bigl\{ i \in [1,s] \colon |X_i| = 2 \bigr\} \bigr| = f (2)$.

\medskip
\noindent
{\it Proof of {\bf A.}} We proceed in two steps.

\noindent
{\bf (i)}  First, we deal with $\mathcal B (G)$.

Let $Z\in \mathsf Z_{\mathcal B (G)} (B) \setminus \{Z_1, \ldots, Z_s\}$.
If $i, j \in [1,s]$ are distinct,   $y_i \in X_i$ and $y_j \in
X_j$, then $p_i^{-1}(y_i) \cap p_j^{-1}(y_j) \ne \emptyset$, and
since $B$ is squarefree, it follows that $B_{p_i^{-1}(y_i)}
B_{p_j^{-1}(y_j)} \nmid B$. Hence there exists some $A \in
\mathcal A \big( \mathcal B(G) \big) \setminus \{ B_{p_i^{-1}(y)} \mid i \in [1,s]\,, \ y
\in X_i \}$ with $A \t Z$. By Lemma \ref{3.4},  this can only hold for
 $\text{char}(R) =2$. We assert also that  $V
= A^{-1}B \in \mathcal A \big( \mathcal B(G) \big)$. If this holds, then we get  $|Z| = 2$, whence we are done.
Indeed, if $V \in \mathcal B(G) \setminus \mathcal A \big( \mathcal B(G) \big)$, then
Lemma \ref{3.4} implies that $V=B_{p_i^{-1}(Y_i)}$ for some $i \in
[1,s]$ and $Y_i \subset X_i$, and consequently $A = V^{-1}B =
B_{p_i^{-1}(X_i \setminus Y_i)}$. But, then Lemma \ref{3.4}  implies that
$|X_i \setminus Y_i| = 1$,  a contradiction to the choice of $A$.

\smallskip
\noindent
{\bf (ii)} Next, we deal with $\mathcal B_{\pm} (G)$, and we may assume that $\textrm{char} (R) \ne 2$.

Let $Z' = A_1 \cdot \ldots \cdot A_m \in \mathsf Z_{\mathcal B_{\pm} (G)} (B) \setminus \{Z_{\pm, 1}, \ldots, Z_{\pm, s} \}$. If $A_1, \ldots, A_m \in \mathcal B (G)$, then they are atoms in $\mathcal B (G)$ by Relation (\ref{2.1}),  whence $Z' \in \mathsf Z_{\mathcal B (G)} (B) \setminus \{Z_1, \ldots, Z_s\}$ and so $|Z'|=2$ by {\bf (i)}. Now suppose that there is $k \in [1,m]$ with $A_k \notin \mathcal B (G)$, say $k=1$. Then $A_2 \cdot \ldots \cdot A_m \in \mathcal B_{\pm} (G) \setminus \mathcal B (G)$. Then Lemmas \ref{3.5} and \ref{3.6} imply that $A_2 \cdot \ldots \cdot A_m \in \mathcal A \big(\mathcal B_{\pm} (G)  \big)$, whence $|Z'|=m=2$.
\qed

\smallskip
\section{On the construction of a group having a given $L$   as a set of lengths} \label{4}
\smallskip
	
The goal in this section is to prove the following result.

\smallskip
\begin{proposition} \label{4.1}
Let $R$ be a commutative ring, let $L \subset \N_{\ge 2}$ be a finite, nonempty subset,  and let $f \colon L \to \N$  be a map. Then there exist
a finitely generated free $R$-module $G$ and some $S \in \mathcal
B (G)$ with the following properties.
\begin{enumerate}
\item[(a)] $S$ is squarefree in $\mathcal F (G)$.

\item[(b)] $\mathsf L_{\mathcal B_{\pm} (G)} (S) = \mathsf L_{\mathcal B (G)} (S) = L$.

\item[(c)] $|\mathsf Z_{\mathcal B_{\pm} (G),k} (S)| \ge |\mathsf Z_{\mathcal B (G),k} (S)| \ge f (k)$ for all $k \in L$. Moreover, both inequalities are equalities for  $k > \min L$.
\end{enumerate}
\end{proposition}

\begin{proof}
We set $s = \sum_{k \in L} f (k)$ and distinguish several cases. Throughout, we use Relation \eqref{atom-inclusion} and Lemma \ref{2.1} without further mention.

\smallskip
\noindent
CASE 1: $s=1$ and $L = \{2\}$.

We set $G = R^4$ and choose  an $R$-basis $(e_1, e_2, f_1, f_2)$ of $G$. Then $U_1 = e_1f_1(-e_1-f_1)$ and $U_2 = e_2f_2(-e_2-f_2)$ are atoms of $\mathcal B (G)$ and of $\mathcal B_{\pm} (G)$. Thus, $S = U_1U_2 \in \mathcal B (G) \subset \mathcal B_{\pm} (G)$ is squarefree with $\mathsf L_{\mathcal B (G)} (S) = \mathsf L_{\mathcal B_{\pm} (G)} (S) = \{2\}$, and with $|\mathsf Z_{\mathcal B (G), 2} (S)| = |\mathsf Z_{\mathcal B_{\pm} (G), 2} (S)| = f(2) = 1$.

\smallskip
\noindent
CASE 2: $s=2$, $L = \{2\}$, and $\text{char} (R) \ne 2$.

We set $G = R^3$ and choose an $R$-basis $(e_1, e_2, e_3)$  of $G$. Then
\[
U_1 = (-e_1)e_1 \,, U_2 = e_2e_3(e_1-e_3)(-e_1-e_2) \,, U_3 = e_1e_2(-e_1-e_2) \,, \quad \text{and} \quad U_4 = (-e_1)e_3(e_1-e_3)
\]
are atoms of  $\mathcal B (G)$ and of $\mathcal B_{\pm} (G)$. Thus, $S = U_1U_2=U_3U_4 \in \mathcal B (G) \subset \mathcal B_{\pm} (G)$ is squarefree with $\mathsf L_{\mathcal B (G)} (S) = \mathsf L_{\mathcal B_{\pm} (G)}  (S) = \{2\}$, and with $|\mathsf Z_{\mathcal B (G), 2} (S)| = |\mathsf Z_{\mathcal B_{\pm} (G), 2} (S)| = f(2) = 2$.

\smallskip
\noindent
CASE 3: $s=2$, $L = \{2\}$, and $\text{char} (R) = 2$.

We set $G = R^4$ and choose an $R$-basis $(e_1, e_2, e_3, e_4)$ of $G$. We define $e_0 = e_1+e_2+e_3+e_4$ and consider
\[
\begin{aligned}
S & = \big( (e_1+e_2)(e_3+e_4)e_0\big) \big( e_1e_2e_3(e_0+e_4)\big) \\
  & = \big( e_1e_2(e_1+e_2)\big) \big( e_3(e_0+e_4)(e_3+e_4)e_0 \big) \\
  & = \big( (e_1+e_2)e_3(e_0+e_4) \big) \big( e_1e_2(e_3+e_4)e_0 \big) \,.
\end{aligned}
\]
Then $S \in \mathcal B (G) = \mathcal B_{\pm} (G)$ is squarefree, $\mathsf L_{\mathcal B (G)} (S) = \mathsf L_{\mathcal B_{\pm} (G)} (S)=2$, and $|\mathsf Z_{\mathcal B (G), 2} (S)| = |\mathsf Z_{\mathcal B_{\pm} (G), 2} (S)| = 3 > 2 = f (2)$.

\smallskip
\noindent
CASE 4: $s=2$, $L = \{2, r\}$ with $r \ge 3$, and $\text{char} (R) \ne 2$.

We set $G=R^{r-1}$ and choose an $R$-basis $(e_1, \ldots , e_{r-1})$  of $G$. Then
 $U = e_0e_1 \cdot \ldots
\cdot e_{r-1}$, with $e_0 = - (e_1+ \ldots + e_{r-1})$,  and $U_i = (-e_i)e_i$, with $i \in [0,r]$, are atoms of $\mathcal B (G)$ and of $\mathcal B_{\pm} (G)$. Thus,
\[
S = (-U)U = U_0 \cdot \ldots \cdot U_{r-1} \in \mathcal B (G) \subset \mathcal B_{\pm} (G)
\]
is squarefree,   $\mathsf  L_{\mathcal B (G)} (S) = \mathsf  L_{\mathcal B_{\pm} (G)} (S) = \{2,r\}$ and $|\mathsf Z_{\mathcal B (G), k} (S)| = |\mathsf Z_{\mathcal B_{\pm} (G), k} (S)| = f (k) = 1$ for every $k \in L$.

\smallskip
\noindent
CASE 5:  $s=2$,  $L = \{2,\,r\}$ with  $r \ge 3$, and $\text{char}(R) = 2$.

We set $G = R^{2r-1}$ and choose an $R$-basis $(e_1, \ldots, e_{2r-1})$  of $G$. With $e_0 = e_1+ \ldots + e_{2r-2}$, the sequences
\[
\begin{aligned}
U_1 & = e_1 \cdot \ldots \cdot e_{2r-1}(e_0+e_{2r-1}) , \quad U_2 = (e_1+e_2)(e_3+e_4) \cdot \ldots \cdot (e_{2r-3}+e_{2r-2})e_0 \,, \\
V_1 & = e_1 \cdot \ldots \cdot e_{2r-2}e_0 , \quad \text{and} \quad V_2 = (e_1+e_2)(e_3+e_4) \cdot \ldots \cdot (e_{2r-3}+e_{2r-2})e_{2r-1}(e_{2r-1}+e_0)
\end{aligned}
\]
are in  $\mathcal B (G) = \mathcal B_{\pm} (G)$. We define
\[
S = U_1U_2=V_1V_2 = \big(e_1e_2(e_1+e_2)\big) \cdot \ldots \cdot \big( e_{2r-3}e_{2r-2}(e_{2r-3}+e_{2r-2})\big)\big(e_{2r-1}e_0(e_{2r-1}+e_0)\big) \,.
\]
By construction, $S \in \mathcal B (G) = \mathcal B_{\pm} (G)$ is squarefree and, obviously, $\mathsf L_{\mathcal B (G)} (S) = \mathsf L_{\mathcal B_{\pm} (G)} (S) = \{2, r\}$, $|\mathsf Z_{\mathcal B (G), 2} (S)| =  |\mathsf Z_{\mathcal B_{\pm} (G), 2} (S)| = 2 > 1 = f (2)$, and $|\mathsf Z_{\mathcal B (G), r} (S)|= |\mathsf Z_{\mathcal B_{\pm} (G), r} (S)|=1=f (r)$.

\smallskip
\noindent
CASE 6:  $s \ge 3$  and  $2 \in L$.

This follows from Proposition \ref{3.1}.

\smallskip
\noindent
CASE 7:    $2 \notin L$.

We set $m = -2+ \min L$,   $L_0 = -m + L \subset \N_{\ge 2}$, and we define
\[
f_0 \colon L_0 \to \N \quad \text{ by} \quad  f_0 (k) = f (k+m) \quad \text{for
every $k \in L_0$} \,.
\]
Since $2 \in L_0$, the previous cases imply that there exist a finitely generated free $R$-module $G_1$ and
some $S_1 \in \mathcal B (G_1) \subset \mathcal B_{\pm} (G_1)$ satisfying Properties (a), (b), and (c). In particular, $\mathsf L_{\mathcal B (G)} (S_1) = \mathsf L_{\mathcal B_{\pm} (G)} (S_1) = L_0$.

Now we set $G =  G_1 \oplus R^{2m}$ and we choose  an $R$-basis $(e_1,\,f_1, \ldots, e_m,\,
f_m)$ of $R^{2m}$. Then
\[
S_2 = \prod_{j=1}^m \,\big( e_if_i(-e_i-f_i) \big) \ \in \ \mathcal B(R^{2m}) \subset \mathcal B_{\pm} (R^{2m})
\]
has unique factorization in $\mathcal B (R^{2m})$ and in $\mathcal B_{\pm} (R^{2m})$, say  $\mathsf  Z_{\mathcal B (R^{2m})} (S_2) = \mathsf  Z_{\mathcal B_{\pm} (R^{2m})} (S_2) = \{y\}$ with $|y|=m$. Now we consider the sequence
\[
S = S_1S_2 \in  \mathcal B (G_1) \time \mathcal B(R^{2m})   \subset \mathcal B (G) \subset \mathcal
B_{\pm} (G)\,.
\]
Clearly, $S$ is squarefree in $\mathcal F (G)$, $\mathsf  Z_{\mathcal B (G)} (S) = \{ yz \colon z \in \mathsf Z_{\mathcal B (G_1)} (S_1)\} $ and  $\mathsf  Z_{\mathcal B_{\pm} (G)} (S) = \{ yz \colon z \in \mathsf Z_{\mathcal B_{\pm} (G_1)} (S_1)\} $, which implies that
\[
\mathsf  L_{\mathcal B_{\pm} (G)} (S) = |y| + \mathsf L_{\mathcal B_{\pm} (G_1)} (S_1) = m+L_0
= L = \mathsf  L_{\mathcal B (G)} (S)
\]
and, for every $k \in L$,
\[
|\mathsf Z_{\mathcal B_{\pm} (G), k} (S)| = |\mathsf Z_{\mathcal B_{\pm} (G_1), k-m}(S_1)| \quad \text{and} \quad
|\mathsf Z_{\mathcal B (G), k} (S)| = |\mathsf Z_{\mathcal B (G_1), k-m}(S_1)|   \,,
\]
Thus, $S \in \mathcal B (G)$ satisfies Properties (a), (b), and (c).
\end{proof}

\smallskip
\section{Proof of Theorem \ref{1.1} and of Corollary \ref{1.2}} \label{5}
\smallskip

For an abelian group $G$, we denote by $\mathsf T (G)$ its torsion subgroup.

\smallskip
\begin{proposition} \label{5.1}
Let $G$ be a finitely generated abelian group of torsionfree rank $r \in \N$ and let $S \in \mathcal B_{\pm} (G)$.
Then there are $N \in \N$ and homomorphisms
\[
\varphi_0 \colon   G \to G_0 := \Z  \times \mathsf T (G) , \qquad \varphi_n
\colon G \to G_n := (\Z/n\Z)^r \times \mathsf T (G)  \quad \text{for all} \ n
\ge N\,,
\]
such that the following properties hold for all $n \in \{0\} \cup \N_{\ge N}$.

\begin{enumerate}
\smallskip
\item If  $S$ is squarefree in $\mathcal F (G)$, then $\varphi_n (S)$ is squarefree in $\mathcal F (G_n)$.

\smallskip
\item The map $\varphi_n$ induces a bijective map  $\overline \varphi_n \colon \mathsf Z_{\mathcal B_{\pm} (G)} (S) \to \mathsf Z_{\mathcal B_{\pm} (G_n)} ( \varphi_n (S) )$ such that the following three properties hold.
    \begin{enumerate}
    \item $|z| = |\overline \varphi_n (z)|$ for all $z \in \mathsf Z_{\mathcal B_{\pm} (G)} (S)$.

    \item $\mathsf L_{\mathcal B_{\pm} (G)} (S) = \mathsf L_{\mathcal B_{\pm} (G_n)} ( \varphi_n (S))$.

    \item  $\bigl| \bigl\{ z \in \mathsf Z_{\mathcal B_{\pm} (G)} (S) \colon |z| = k \bigr\} \bigr|
           = \bigl| \bigl\{ z \in \mathsf Z_{\mathcal B_{\pm} (G_n)} ( \varphi_n (S)) \colon |z| = k \bigr\} \bigr|$
              for every $k \in \mathsf L_{\mathcal B_{\pm} (G)} (S)$.
    \end{enumerate}

\smallskip
\item If $S \in \mathcal B (G)$, then     $\varphi_n$ induces a bijective map  $\varphi_n' \colon \mathsf Z_{\mathcal B (G)} (S) \to \mathsf Z_{\mathcal B (G_n)} ( \varphi_n (S) )$ such that the \newline Properties (a) - (c) of Part 2 hold for $\mathcal B (G)$ and $\mathcal B (G_n)$.
\end{enumerate}
\end{proposition}

\begin{proof}
We may suppose that $G = \Z^r \times \mathsf T (G)$ and let $\mathsf p \colon G \to \Z^r$ denote the projection.
We set $S = g_1 \cdot\ldots \cdot g_{\ell}$, where $\ell \in \N_0$
and $g_1, \ldots, g_{\ell} \in G$. It  suffices to show that there
are $N \in \N$ and   homomorphisms $\varphi_n \colon G \to (\Z/n\Z)^r \times \mathsf T (G)$, for every $n \in \{0\} \cup \N_{\ge N}$,  which are injective on the set
\[
\Sigma (S) = \Bigl \{ \,\sum_{\nu \in I}  g_{\nu}\; \colon \;
\emptyset \ne I \subset [1,\ell] \,,  \ \text{for all} \ \nu \in I \Bigr\} \,,
\]
which settles Part 3., and on the set
\[
\Sigma_{\pm} (S) = \Bigl \{ \,\sum_{\nu \in I} \varepsilon_{\nu} g_{\nu}\; \colon \;
\emptyset \ne I \subset [1,\ell] \,, \varepsilon_{\nu} \in \{-1,1\} \ \text{for all} \ \nu \in I \Bigr\} \,,
\]
which settles Part 2. (in case when $S \in \mathcal B (G)$).
Clearly, $\Sigma (S) \subset \Sigma_{\pm} (S)$ and the set
\[
E = \big\{ \boldsymbol m - \boldsymbol n \colon \boldsymbol m, \,
\boldsymbol n \in \mathsf p (\Sigma_{\pm} (S)), \ \boldsymbol m \ne
\boldsymbol n \big\}  \subset\Z^r \setminus \{ \boldsymbol 0 \}
\]
is finite, say $E= \{ \boldsymbol n^{(1)}, \ldots, \boldsymbol n^{(d)} \}$. We assert that there is a homomorphism $\psi \colon \Z^r \to \Z$ such that $\psi (\boldsymbol n^{(\nu)} ) \ne 0$ for all $\nu \in [1,d]$. If this holds, then we set $\varphi_0 \colon G \to G_0$ by $\varphi_0 (\boldsymbol u, g) =(\psi ( \boldsymbol u ), g)$, and obviously $\varphi_0$ has the required property.

For $j \in [1,d]$, let $\boldsymbol n^{(j)} = (n_1^{(j)}, \ldots, n_r^{(j)})$, and consider the non-zero polynomial
\[
f = \prod_{j =1}^d \Bigl( \sum_{i=1}^r n_i^{(j)} X_i \Bigr) \in \Z[X_1, \ldots, X_r] \,.
\]
If $(a_1, \ldots, a_r) \in \Z^r$ is such that $f(a_1, \ldots, a_r)\ne 0$, then $\psi \colon \Z^r \to \Z$, defined by $\psi (x_1, \ldots, x_r) = a_1x_1 + \ldots + a_rx_r$, is a non-zero homomorphism satisfying $\psi (\boldsymbol n^{(j)} ) \ne 0$ for all $j \in [1,d]$.

\smallskip

For $n \in \N$, let $\psi_n \colon \Z^r \to (\Z/n\Z)^r$ be the
canonical epimorphism. Then there exists some $N \in \N$ such
that $\psi_n \t \mathsf p (\Sigma_{\pm} (S))$ is injective for all $n \ge N$
and the homomorphism $\varphi_n = \psi_n \times \id_{\mathsf T (G)} \colon G \to G_n$ has the required property.
\end{proof}

\smallskip
We recall the definition of almost arithmetic multiprogressions (AAMPs) and of the set of minimal distances of a monoid. To begin with AAMPs, let   $d\in\N$, $M\in\N_0$ and $\{0,d\}\subset\mathcal D
\subset [0,d]$. Then $L$ is called an {\it almost
arithmetic multiprogression}\ ({\rm AAMP}\ for
short) with {\it difference} $d$, {\it period}\ $\mathcal D$,
 and {\it bound} $M$, if
\[
L = y + (L'\cup L^*\cup L'')\,\subset\,y +\mathcal D + d\Z
\]
where
\begin{itemize}
\item $L^*$ is finite and nonempty with $\min L^* = 0$ and $L^* =(\mathcal D + d\Z)\cap [0,\max L^*]$,
\item $L'\subset [-M,-1]$ and $L''\subset\max L^* + [1,M]$, which are called the {\it initial part} and the {\it end part} of $L$, and
\item $y\in\Z$.
\end{itemize}
In particular,  AAMPs are finite and nonempty subsets of the integers. For an overview of monoids and domains, whose sets of lengths are AAMPs with global bounds on all parameters, we refer to \cite[Chapter 4]{Ge-HK06a}. Schmid proved that for Krull monoids with finite class group the description of sets of lengths as AAMPs is best possible (\cite{Sc09a}).

\smallskip
Next we recall the concept of minimal distances.
For a finite, nonempty subset $L = \{m_0, \ldots, m_k \} \subset \Z$ with $m_0 < \ldots < m_k$, let $\Delta (L) = \{m_i - m_{i-1} \colon i \in [1, k] \} \subset \N$ denote its set of distances. Now let $H$ be a \BF-monoid. Then every divisor-closed submonoid $S\subset H$ is a BF-monoid and
\[
\Delta (S) = \bigcup_{a \in S} \Delta \big( \mathsf L_S (a) \big) \subset \N
\]
denotes the {\it set of distances} of $S$. Then
\[
\Delta^* (H) = \{ \min \Delta (S) \colon S \subset H \ \text{is a divisor-closed submonoid with $\Delta (S) \ne \emptyset$} \} \subset \Delta (H)
\]
is the {\it set of minimal distances} of $H$.  In the next lemma, we gather some simple properties, which highlight the differences between $\Delta^* \big( \mathcal B (G) \big)$ and $\Delta^* \big( \mathcal B_{\pm} (G) \big)$ and which shows how different $\Delta^* \big( \mathcal B_{\pm} (G) \big)$ can be for different groups. Let
\[
\mathsf D \big( \mathcal B_{\pm} (G) \big) = \max \big\{ |S| \colon S \in \mathcal A \big( \mathcal B_{\pm} (G) \big) \big\}
\]
denote the {\it plus-minus weighted Davenport constant} of $G$, which is studied, among others, in \cite{Ma-Or-Sc14a,Ma-Or-Sa-Sc15,HK14b}.

\smallskip
\begin{proposition} \label{5.2}
Let $G$ be an abelian group.
\begin{enumerate}
\item If $G$ is finite with $|G|>2$, exponent $n \ge 2$, and rank $r \ge 1$, then  $\max \Delta^* \big( \mathcal B (G) \big) = \max \{r-1, n-2\}$ and $\max \Delta^* \big( \mathcal B_{\pm} (G) \big) \le \mathsf D \big( \mathcal B_{\pm} (G) \big)-2$.

\item If $G$ is infinite, then $\Delta^* \big( \mathcal B (G) \big) = \N$.

\item If $G$ is a finite elementary $2$-group of rank $r$, then $\Delta^* \big( \mathcal B (G) \big) = \Delta^* \big( \mathcal B_{\pm} (G) \big) = [1, r-1]$.

\item If $G$ is an elementary $3$-group (finite or infinite), then $\Delta^* \big( \mathcal B_{\pm} (G) \big) = \{1\}$.
\end{enumerate}
\end{proposition}

\begin{proof}
1. For the first statement see \cite{Ge-Zh16a},  and for the second see \cite{Me-Or-Sc25a}.

2.  See \cite[Theorem 1.1]{Ch-Sc-Sm08b}.

3. Since $\mathcal B (G) = \mathcal B_{\pm} (G)$, we have $\Delta^* \big( \mathcal B (G) \big) = \Delta^* \big( \mathcal B_{\pm} (G) \big)$. The equality $\Delta^* \big( \mathcal B (G) \big)  = [1, r-1]$ follows from \cite[Corollary 6.8.3]{Ge-HK06a}.

4. Let $S \subset \mathcal B_{\pm} (G)$ be a divisor-closed submonoid with $\Delta (S) \ne \emptyset$. Then there is a nonempty subset $G_0 \subset G$ with $\mathcal B_{\pm} (G_0) = S$. Let $g \in G_0$ be a nonzero element. Then $\ord (g) = 3$, $A = g^3$ and $U = g^2$ are atoms in $\mathcal B_{\pm} (G_0)$. Since $A^2 = U^3$, it follows that $1 \in \Delta (S)$, whence $\min \Delta (S)=1$.
\end{proof}

We refer to  \cite{Ge-Zh16a, Zh18a,Pl-Sc19a} for  recent progress on $\Delta^* \big( \mathcal B (G) \big)$. For the monoid $\mathcal B_{\pm} (G)$, its set of minimal distances was recently studied by Schmid et al. \cite{Me-Or-Sc25a}, and it turned out, as already indicated by Proposition \ref{5.2},  that the structure of $\Delta^* \big( \mathcal B_{\pm} (G) \big)$ is in general quite different from the structure of $\Delta^* \big( \mathcal B (G) \big)$.

\smallskip
Before completing the proofs of Theorem \ref{1.1} and of Corollary \ref{1.2}, let us compare the statements of Part 1 and of Part 2 of Theorem \ref{1.1} for elementary $3$-groups. Let $G$ be an elementary $3$-group. If $G \cong C_3^r$ with $r \in \N$, then, apart from a globally bounded initial and end part,  all sets of lengths in $\mathcal B_{\pm} (G)$ are intervals by Theorem \ref{1.1}.1 and by  Proposition \ref{5.2}.4. Nevertheless, if $G$ is infinite, then every finite, nonempty subset of $\N_{\ge 2}$ occurs as a set of lengths. This shows that the bounds $M (C_3^r)$, occurring in Theorem \ref{1.1}.1, tend to infinity as $r$ tends to infinity.

\smallskip
An abelian group $G$ is said to be {\it bounded} if there is $m \in \N$ such that $mG= \boldsymbol 0$.

\smallskip
\begin{proof}[Proof of Theorem \ref{1.1}]
Let $G$ be an abelian group.

1. By \cite[Theorem 3.7]{F-G-R-Z24a}, $\mathcal B_{\pm} (G)$ is finitely generated if and only if $G$ is finite. Thus, if $G$ is finite, then Theorem 4.4.11 in \cite{Ge-HK06a} implies that there is some $M \in \N_0$ such that, for every plus-minus weighted zero-sum sequence $S$ over $G$, its set of lengths $\mathsf L_{\mathcal B_{\pm} (G)} (S)$ is an AAMP with difference in $\Delta^* \big( \mathcal B_{\pm} (G) \big)$ and bound $M$.

2. Now suppose that $G$ is infinite. Let $L \subset \N_{\ge 2}$ be a finite, nonempty subset and let $f \colon L \to \N$ be a map. We show that there is an abelian group $G_1$, which is isomorphic to a subgroup of $G_2 \subset G$, and some $S \in \mathcal B (G_1)$ having the required properties. If this holds, then the assertion follows because $\mathcal B_{\pm} (G_2)$ is a divisor-closed submonoid of $\mathcal B_{\pm} (G)$.

We distinguish three cases.

\smallskip
\noindent
CASE 1:  $G$ is not a torsion group.

Then $G$ has a subgroup isomorphic to $\Z$. By Proposition
\ref{4.1}, there exist some $r \in \N$ and some $S' \in \mathcal
B (\Z^r)$ such that $S'$ satisfies the Properties (a), (b), and (c) in Proposition \ref{4.1}.
By Proposition \ref{5.1},  there exists some $S
\in \mathcal B ( \Z)$ satisfying the same Properties (a), (b), and (c).

\smallskip
\noindent
CASE 2:  $G$ is an unbounded torsion group.

By CASE 1, there exists some $S' \in \mathcal B (\Z)$ satisfying the Properties (a), (b), and (c) in Proposition \ref{4.1}. By Proposition \ref{5.1},  there exists some $N \in
\N$ such that for every $n \ge N$ there is some $S_n \in \mathcal
B_{\pm} (\Z/ n \Z)$ with the required properties.
Since $G$ contains a cyclic subgroup of order $n$ for some $n \ge N$, the assertion follows.

\smallskip
\noindent
CASE 3:  $G$ is bounded.

By \cite[Chapter 4]{Ro96}, $G$ is a direct sum of cyclic groups.
Hence  there is some $n \in \N_{\ge 2}$ such that $G$ contains a
subgroup isomorphic to $(\Z/n\Z)^{(\N)}$.  In particular, for every $r \in \N$,  $G$
contains a subgroup isomorphic to $(\Z/n\Z)^r$, whence the assertion
follows by Proposition \ref{4.1}.
\end{proof}

\smallskip
\begin{proof}[Proof of the Corollary \ref{1.2}]
Let $L \subset \N_{\ge 2}$ be a finite, nonempty subset. For a prime $p$ and $r \in \N$, let $C_p^r$ denote an elementary $p$-group of rank $r$. We prove the following two claims.
\begin{itemize}
\item[{\bf C1}] There is $N \in \N$ such that $L \in \mathcal L \big( \mathcal B_{\pm} (\Z/n\Z) \big)$ for all $n \ge N$.

\item[{\bf C2}] For every prime $p$ there is $s_p \in \N$ such that $L \in \mathcal L \big( \mathcal B_{\pm} (C_p^s) \big)$ for all $s \ge s_p$.
\end{itemize}
Suppose that {\bf C1} and {\bf C2} hold and let $G$ be a finite abelian group with $L \notin \mathcal L \big( \mathcal B_{\pm} (G) \big)$. Then {\bf C1} shows that $\exp (G)  < N$ and {\bf C2} implies that the $p$-rank of $G$ is bounded above by $s_p$ for all primes $p$ dividing $\exp (G)$. Thus, the assertion follows.

\smallskip
\noindent
{\it Proof of {\bf C1.}} By Theorem \ref{1.1}, there is some $S = m_1 \cdot \ldots \cdot m_{\ell} \in \mathcal B ( \Z )$ such that $L = \mathsf L_{\mathcal B_{\pm} ( \Z )} (S)$. We set $N = 1+\sum_{i=1}^{\ell} |m_i|$. Let  $n \in \N$ with $n \ge N$ and let $\varphi_n \colon \Z \to \Z/n\Z$ denote the canonical epimorphism. If $A \in \mathcal F (\Z)$ with $A \t S$ in $\mathcal F ( \Z)$, then $A \in \mathcal B_{\pm} ( \Z)$ if and only if $\varphi_n (A) \in \mathcal B_{\pm} ( \Z/n\Z)$. This implies that $\mathsf L_{\mathcal B_{\pm} ( \Z )}(S) = \mathsf L_{\mathcal B_{\pm} ( \Z/n\Z )}(\varphi_n (S))$.

\smallskip
\noindent
{\it Proof of {\bf C2.}} Let $p$ be a prime and let $G$ be an elementary $p$-group of infinite rank. By Theorem \ref{1.1},  there is some $S \in \mathcal B(G)$ such that $\mathsf L_{\mathcal B_{\pm} (G)} (S) = L$. Clearly, $G_1= \langle \supp (S) \rangle \subset G$ is an elementary $p$-group of finite rank, $S \in \mathcal B (G_1)$, and  $L = \mathsf L_{\mathcal B_{\pm} (G_1)} (S) $. Thus, {\bf C2} holds with $s_p$ being the $p$-rank of $G_1$.
\end{proof}

\providecommand{\bysame}{\leavevmode\hbox to3em{\hrulefill}\thinspace}
\providecommand{\MR}{\relax\ifhmode\unskip\space\fi MR }
\providecommand{\MRhref}[2]{%
  \href{http://www.ams.org/mathscinet-getitem?mr=#1}{#2}
}
\providecommand{\href}[2]{#2}

\end{document}